\documentclass[a4paper]{article}

\usepackage{authblk}
\usepackage{amsmath}
\usepackage{amssymb}
\usepackage{amsthm}
\usepackage{mathtools}
\usepackage{bm}
\usepackage{bbm}
\usepackage{a4wide}
\usepackage{enumitem}
\usepackage{hyperref}
\usepackage{booktabs}

\usepackage{tikz}
\usepackage{graphicx}
\usepackage{todonotes}

\newcommand{\sss}{\scriptscriptstyle}

\newcommand{\prob}{\mathbb{P}}
\newcommand{\Prob}[1]{\prob\left(#1\right)}

\newcommand{\CProb}[2]{\prob\left(\left.#1\right|#2\right)}

\newcommand{\expec}{\mathbb{E}}
\newcommand{\Exp}[1]{\expec\left[#1\right]}

\newcommand{\Var}[1]{\textup{Var}\left(#1\right)}

\newcommand{\plim}{\ensuremath{\stackrel{\prob}{\longrightarrow}}}
\newcommand{\dlim}{\ensuremath{\stackrel{d}{\longrightarrow}}}

\newcommand{\ind}[1]{\mathbbm{1}_{\left\{#1\right\}}}

\newcommand{\bigO}[1]{O\left(#1\right)}
\newcommand{\bigOp}[1]{O_{\sss\prob}\left(#1\right)}

\newcommand{\me}{\textup{e}}
\newcommand{\dd}{{\rm d}}

\allowdisplaybreaks

\newtheorem{theorem}{Theorem}

\newtheorem{lemma}[theorem]{Lemma}

\newtheorem{assumption}{Assumption}

\theoremstyle{definition}
\newtheorem{remark}{Remark}

\newcommand{\op}{o_{\sss\prob}}
\newcommand{\Nn}{N^{\sss(n)}}
\newcommand{\Mn}{M^{\sss(n)}}

\setenumerate{label={\normalfont{(\roman*)}}} 
\SetLabelAlign{LeftAlignWithIndent}{\hspace*{1.0ex}\makebox[1.25em][l]{#1}}
\numberwithin{equation}{section}

\title{Counting triangles in power-law uniform random graphs}
\author[1]{Pu Gao}
\author[2]{Remco van der Hofstad}
\author[1]{Angus Southwell}
\author[2]{Clara Stegehuis}
\affil[1]{Monash University}
\affil[2]{Eindhoven University of Technology}

\graphicspath{{Figures/}}
\begin{document}
	
	\maketitle
	\begin{abstract}
		We count the asymptotic number of triangles in uniform random graphs where the degree distribution follows a power law with degree exponent $\tau\in(2,3)$. 
		We also analyze the local clustering coefficient $c(k)$, the probability that two random neighbors of a vertex of degree $k$ are connected. We find that the number of triangles, as well as the local clustering coefficient, scale similarly as in the erased configuration model, where all self-loops and multiple edges of the configuration model are removed. Interestingly, uniform random graphs contain more triangles than erased configuration models with the same degree sequence. The number of triangles in uniform random graphs is closely related to that in a version of the rank-1 inhomogeneous random graph, where all vertices are equipped with weights, and the edge probabilities are moderated by asymptotically linear functions of the products of these vertex weights.
	\end{abstract}
	
	\section{Introduction}
	Many real-world networks were found to have degree distributions that can be well approximated by a power-law distribution, so that the fraction of vertices of degree $k$ scales as $k^{-\tau}$ for some $\tau>1$. The degree exponent $\tau$ of several networks was found to satisfy $\tau\in(2,3)$~\cite{vazquez2002}. These power-law real-world networks are often modeled by random graphs. One very natural null model for real-world networks is the uniform random graph with prescribed degrees, that samples a simple graph uniformly from all graphs with the same degree sequence as the original network. By definition, the resulting graph has the same degree sequence as the original network, but whether other properties behave similarly in real-world networks and uniform random graphs is an interesting question. In this paper, we focus on the property of triangle counts. Triangle counts measure the tendency of two neighbors of a vertex to be connected as well, allowing to analyze clustering properties of real-world networks.

	Uniform random graphs with specified degree sequences are among the most commonly studied random graphs models~\cite{wormald1999,molloy1995,janson2009}. Compared with the classical Erd\H{o}s-R\'{e}nyi graphs, analyzing uniform random graphs with specified degree sequences is much more challenging. The probabilities of simple events, such as two given vertices being adjacent, are highly non-trivial to be estimated to a desired accuracy. The most commonly used method for studying these random graphs is to use the configuration model introduced by Bollob\'{a}s~\cite{B1980} (see also~\cite[Chapter 7]{hofstad2009}). The configuration model constructs a random multigraph with a specified degree sequence. Conditioning on the resulting graph being simple, the distribution is uniform. Estimating edge probabilities in the configuration model is easy and thus we may analyze the configuration model and then translate the result to uniform random graphs. Such a translation is possible only if the probability of producing a simple graph by the configuration model is not too small. However, with a power-law degree sequence with exponent $\tau<3$, this probability vanishes~\cite{janson2014,bollobas2015}. Thus, analyzing the configuration model does not help in analyzing uniform random graphs when $\tau\in(2,3)$. 
	
	Other methods for analyzing uniform random graphs rely on asymptotic enumeration of graphs with specified degree sequences. The switching method is useful in enumerating sparse graphs~\cite{mckay1981}, whereas enumerating dense graphs uses multidimensional complex Cauchy integrals and the Laplace method~\cite{mckay1990}. Studies of subgraphs of uniform graphs with specified degree sequences are based on afore-mentioned enumeration results and techniques~\cite{mckay2011,greenhill2018}.
	
	As a compromise, a commonly used practice for $\tau\in(2,3)$ is to use the erased configuration model~\cite{britton2006,berry2015,schlauch2015}. 
	Instead of conditioning on simple graphs, the erased configuration model generates a random multigraph using the configuration model, and then removes all loops and replaces multiple edges by simple edges. Graphs in the erased configuration model are not uniformly distributed. It was believed that many properties can be translated between random graphs in the erased configuration model and the uniform model, although theoretically this is hard to verify. However, with power-law exponent $\tau<3$, there are edge probabilities which differ significantly in the erased configuration model and in the uniform graph model (see Remark~\ref{rem:edgeprob}).
%
%
%
%
	
Another method to generate graphs with approximately the given degree sequence is to use rank-1 inhomogeneous random graphs~\cite{chung2002,boguna2003}. In rank-1 inhomogeneous random graphs, each vertex $i$ is equipped with a weight $w_i$, and pairs of vertices are connected independently with probability $p(w_i,w_j)$ for some function $p(w_i,w_j)$. Under suitable choices for $p(w_i,w_j)$, the expected degree of a vertex with weight $w_i$ is approximately, or equals $w_i$. Thus, on average, every vertex has degree equal to its targeted degree. However, in general the degree sequence of the inhomogeneous random graph does not equal the targeted degree sequence. In particular, the inhomogeneous random graph typically contains a linear number of isolated vertices. 
	
One version of the rank-1 inhomogeneous random graph is the generalized random graph~\cite{britton2006}, where $p(w_i,w_j)=w_iw_j/(w_iw_j+\sum_sw_s)$. While this model has significantly different properties than the uniform random graph, as described above, we show that the number of triangles in the generalized random graph and the uniform random graph behaves similarly.
	
	In this paper, we focus on counting triangles in sparse uniform random graphs when $\tau\in(2,3)$. When the maximum degree is bounded, or grows slowly, the number of triangles as well as other short cycles in uniform random graphs are asymptotically Poisson distributed~\cite{bollobas1980,wormald1981,mckay2004}. However, when the degree exponent satisfies $\tau\in(2,3)$, the maximum degree grows as fast as $n^{1/(\tau-1)}$, so that the Poisson limit for the number of triangles no longer holds.  
	We count the number of triangles in two steps. First we show that the main contribution to the number of triangles is from vertices with degrees proportional to $\sqrt{n}$. Thus, even though the maximal degree may be much higher than $\sqrt{n}$, vertices with these high degrees are so rare that they can be neglected when counting triangles. From there, we can use a switching method to count the number of triangles between vertices of degrees proportional to $\sqrt{n}$, resulting in an asymptotic expression for the number of triangles in a uniform random graph with power-law degree exponent $\tau\in(2,3)$.
	
	We then proceed to count triangles where one vertex is constrained to have degree $k$. This allows us to investigate the local clustering coefficient $c(k)$, the probability that two random neighbors of a randomly chosen vertex of degree $k$ are connected. 
	Again, we show that the contribution to $c(k)$ from vertices with degrees outside a specific range is small, and use a switching argument to count the number of constrained triangles from vertices with degrees inside the specified ranges. We show that the $k\mapsto c(k)$ curve consists of three regimes. First, the curve remains flat, then it starts to decay logarithmically in $k$, and finally it decays as a power of $k$. This decay of $c(k)$ as a power of $k$ was also observed in several real-world networks~\cite{vazquez2002,maslov2004,krioukov2012}. 
	
	\paragraph{Notation.}
	We use $\plim $ for convergence in probability. We write $[n]=\{1,\dots,n\}$. We say that a sequence of events $(\mathcal{E}_n)_{n\geq 1}$ happens with high probability (w.h.p.) if $\lim_{n\to\infty}\Prob{\mathcal{E}_n}=1$. Furthermore, we write $f(n)=o(g(n))$ if $\lim_{n\to\infty}f(n)/g(n)=0$, and $f(n)=O(g(n))$ if $f(n)/g(n)$ is uniformly bounded, where $(g(n))_{n\geq 1}$ is nonnegative.
	We say that $X_n=O_{\sss{\prob}}(g(n))$ for a sequence of random variables $(X_n)_{n\geq 1}$ if for any $\varepsilon>0$ there exists $M_\varepsilon>0$ such that $\Prob{|X_n|/g(n)>M_\varepsilon}<\varepsilon$, and $X_n=o_{\sss{\prob}}(g(n))$ if $X_n/g(n)\plim 0$.
	
	\subsection{Uniform random graphs} Given a positive integer $n$ and a \emph{degree sequence}, i.e., a sequence of $n$ positive integers $\boldsymbol{d}=(d_1,d_2,\ldots, d_n)$, where $\sum_{i=1}^n d_i\equiv 0\pmod 2$, the \emph{uniform random graph} is a simple graph, uniformly sampled from the set of all simple graphs with degree sequence $(d_i)_{i\in[n]}$. Here we always assume that $\boldsymbol{d}$ is a realizable degree sequence, meaning that there exists a simple graph with degree sequence $\boldsymbol{d}$. Let $\mathcal{G}(\boldsymbol{d})$ denote the ensemble of all simple graphs on degree sequence $\boldsymbol{d}$, and let $d_{\max}=\max_{i\in[n]}d_i$ and $L_n=\sum_{i=1}^n d_i$. We denote the empirical degree distribution by
	\begin{equation}
	F_n(j)=\frac{1}{n}\sum_{i\in[n]}\ind{d_i\leq j}.
	\end{equation}
	
	We study the setting where the variance of $\boldsymbol{d}$ diverges when $n$ grows large.	
	In particular, we assume that the degree sequence satisfies the following assumption:
	\begin{assumption}[Degree sequence]\label{ass:degrees}
		\leavevmode
				\begin{enumerate}
								\item \label{ass:degreeall}
					There exist $\tau\in(2,3)$ and a constant $K>0$ such that for every $n\ge 1$ and every $0\le j\le d_{\max}$, 
					\begin{equation}\label{eq:bound}
					1-F_n(j)\leq K j^{1-\tau}.
					\end{equation}
			\item \label{ass:degreerange}
			There exist $\tau\in(2,3)$ and a constant $C>0$ such that, for all $j=O(\sqrt{n})$,
			\begin{equation}\label{D-tail}
				1-F_n(j) = Cj^{1-\tau}(1+o(1)).
			\end{equation}
		\end{enumerate}
	\end{assumption}
It follows from~(\ref{eq:bound}) that
\begin{equation}\label{dmax}
d_{\max} < M n^{1/(\tau-1)},\quad \mbox{for some sufficiently large constant $M>0$.}
\end{equation}

Assumption~\ref{ass:degrees}\ref{ass:degreerange} is more detailed than Assumption~\ref{ass:degrees}\ref{ass:degreeall}. Assumption~\ref{ass:degrees}\ref{ass:degreeall} states that for all $j$ the inverse cumulative distribution function is bounded from above by some power law. Assumption~\ref{ass:degrees}\ref{ass:degreerange} then states that a pure power-law degree distribution holds for a smaller range of degrees. If we denote by $D_n$ a uniformly chosen degree in $\boldsymbol{d}$, then Assumption~\ref{ass:degrees}\ref{ass:degreerange} implies that $D_n\dlim D$, where 
\begin{center}
$D$ is the random variable with inverse cumulative distribution function $1-F(j)=Cj^{1-\tau}$. 
\end{center}
Note that $\Var{D}=\infty$ if $\tau\in(2,3)$. Moreover, it is easy to see that Assumption~\ref{ass:degrees}(i) and (ii) imply $\Exp{D_n}\to \Exp{D}$ and thus
\[
L_n=(1+\op(1))\mu n,\quad \mbox{where $\mu=\Exp{D}$.}
\]

To prove our results on the number of triangles, Assumption~\ref{ass:degrees} is sufficient. To investigate the local clustering coefficient $c(k)$ over the entire range of $k$, a more detailed assumption on the degree sequence is necessary:
	\begin{assumption}[Degree sequence, stronger assumption]\label{ass:degreesck}
	\leavevmode
		Assumption~\ref{ass:degrees}\ref{ass:degreeall} holds, and furthermore
		\begin{enumerate}[label=\textup{(ii)'}]
		\item \label{ass:degreerangeck}
		There exist $\tau\in(2,3)$ and constants $C,c>0$ such that for all $j\leq cn^{1/(\tau-1)}/\log(n)$,
		\begin{equation}\label{D-tailck}
		1-F_n(j)=Cj^{1-\tau}(1+o(1)).
		\end{equation}
	\end{enumerate}
\end{assumption}
Note that an i.i.d.\ sample from a power-law distribution with exponent $\tau$ satisfies Assumption~\ref{ass:degrees}\ref{ass:degreeall} and Assumption~\ref{ass:degreesck}\ref{ass:degreerangeck} with high probability.



\subsection{Outline}
We first describe our main results on triangle counts as well as the local clustering coefficient in Section~\ref{sec:results}. We then use a switching argument in Section~\ref{sec:conprob} to obtain the connection probabilities between two vertices conditionally on a finite number of edges being present. We  prove our result for the number of triangles in Section~\ref{sec:prooftriang} and for $c(k)$ in Section~\ref{sec:proofck}. Finally, we provide a conclusion in Section~\ref{sec:conc}.

\section{Main results}\label{sec:results}
We now describe the results for the number of triangles as well as the local clustering coefficient in graphs sampled uniformly from $G(\boldsymbol{d})$.
	\subsection{Number of triangles}
	Let $T(G)$ denote the number of triangles in graph $G$. 
	Then, the following result holds for $T(G)$: 
	\begin{theorem}[Number of triangles]\label{thm:triang}
			Let $\tau\in (2,3)$ and $\boldsymbol{d}_n$ be a degree sequence on $n$ vertices satisfying Assumption~\ref{ass:degrees}. Let $G_n$ be a random graph in $\mathcal{G}(\boldsymbol{ d}_n)$, $\mu=\Exp{D}$ and $C$ be the constant in~\eqref{D-tail}. Then, 
			\begin{equation}
			\frac{T(G_n)}{n^{\frac{3}{2}(3-\tau)}}\plim \frac{1}{6}(C(\tau-1))^3\mu^{-\frac{3}{2}(\tau-1)}\!\!\int_{0}^{\infty}\int_{0}^{\infty}\int_{0}^{\infty}\!\!\frac{(xyz)^{2-\tau}}{(1+xy)(1+yz)(1+xz)}\dd x \dd y \dd z<\infty .
			\end{equation}
	\end{theorem}

\emph{Comparison with the erased configuration model.}
The result on the number of triangles is very similar to the number of triangles in the erased configuration model~\cite{hofstad2017}, where all multiple edges of the configuration model are merged and all self-loops are removed. In the erased configuration model instead,
	\begin{align}
	& \frac{T(G_n)}{n^{\frac{3}{2}(3-\tau)}}\plim\nonumber\\
	&  \frac{1}{6}(C(\tau-1))^3\mu^{-\frac{3}{2}(\tau-1)}\! \int_{0}^{\infty}\! \int_{0}^{\infty}\! \int_{0}^{\infty}\! (xyz)^{-\tau}(1-\me^{-xy})(1-\me^{-yz})(1-\me^{-xz})\dd x \dd y \dd z<\infty .
	\end{align}
	Note that $1-\me^{-x}\geq x/(1+x)$ for all $x>0$. 
	Interestingly, this implies that the erased configuration model contains \emph{more} triangles than the uniform random graph with the same degree sequence, even though edges are removed in the erased configuration model, due to the presence of multiple edges and loops in the configuration model, which was empirically observed in~\cite{bannink2018}. 
This is because the large-degree vertices in the uniform random graph model have more low-degree neighbours than in the erased configuration model. These low-degree vertices barely participate in triangles. 
	\\~\\
	\emph{Similarity to generalized random graphs.}
	In generalized random graphs~\cite{britton2006}, every vertex $i$ is equipped with a weight $w_i$, where the weight sequence is an i.i.d.\ sample of~\eqref{D-tail}. A pair of vertices $i$ and $j$ is then connected with probability 
	\begin{equation}
	\Prob{i\sim j\mid (w_s)_{s\in[n]}}=\frac{w_iw_j}{\sum_{s\in[n]}w_s+w_iw_j},
	\end{equation}
	independently for all pairs of vertices.
The probability that a triangle between vertices $i,j$ and $k$ is present can then be written as
		\begin{equation}
	\Prob{\triangle_{i,j,k}\mid (w_s)_{s\in[n]}}=\frac{w_iw_j}{\sum_{s\in[n]}w_s+w_iw_j}\frac{w_iw_k}{\sum_{s\in[n]}w_s+w_iw_j}\frac{w_jw_k}{\sum_{s\in[n]}w_s+w_iw_j}.
	\end{equation}
	Conditionally on the degree sequence of the generalized random graph, the resulting graph is a uniform random graph on that degree sequence. Thus, to prove that Theorem~\ref{thm:triang} also holds for generalized random graphs, one only needs to show that the degree sequence obtained by the generalized random graph satisfies Assumption~\ref{ass:degrees} with high probability, which is shown in~\cite[Chapter 7]{hofstad2009}.
	\\~\\
	\emph{$\sqrt{n}$ degrees.}
	In the proof of Theorem~\ref{thm:triang}, we show that the main contribution to the number of triangles is from vertices of degrees proportional to $\sqrt{n}$. By a switching argument, we show that the probability of a triangle being present between vertices of degrees proportional to $\sqrt{n}$ is asymptotically bounded away from 0 and 1. We show that the probability that a triangle is present between vertices of degrees much lower than $\sqrt{n}$ tends to zero in the large-network limit, which intuitively explains why vertices of degree scaling smaller than $\sqrt{n}$ have a low contribution to the number of triangles. Because of the power-law distribution, vertices of degree much higher than $\sqrt{n}$ are more rare, and therefore do not contribute much to triangle counts. The main contribution of $\sqrt{n}$ vertices also explains why we need the pure power-law degree distribution of Assumption~\ref{ass:degrees}\ref{ass:degreerange} to hold only for vertices of degrees at most proportional to $\sqrt{n}$. Indeed, we show that vertices with higher degrees barely contribute to the triangle count using only the power-law upper bound of Assumption~\ref{ass:degrees}\ref{ass:degreeall}, so that the pure power-law assumption is not necessary for high-degree vertices.
	\\~\\
	\emph{Counting other subgraphs.}
	While our results are for triangle counts, our method easily extends to counting several other types of subgraphs. To prove our theorem on the number of triangles, we mainly use that the main contribution to the number of triangles is from vertices of degrees proportional to $\sqrt{n}$. In~\cite{hofstad2017d}, it was shown that in the erased configuration model, the class of subgraphs where the main contribution is from vertices of degrees proportional to $\sqrt{n}$ is wider, containing for example also all complete graphs of larger sizes. It is easy to show that this class of subgraphs is the same for uniform random graphs, enabling to analyze these subgraph counts in a very similar manner as triangle counts. We do not elaborate on such results and refer to~\cite[Theorem 2.1]{hofstad2017d} which apply here too with the change that $1-\me^{-x}$ should be replaced with $ x/(1+x)$ in all results.
	
		\subsection{Local clustering coefficient}
		We now investigate the triangle structure in uniform random graphs in more detail. 
	Let $\triangle_k$ denote the number of triangles attached to vertices of degree $k$ in the uniform random graph, and when a triangle contains three degree-$k$ vertices it is counted three times. When a triangle consists of two vertices of degree $k$, it is counted twice in $\triangle_k$.  Let $N_k$ denote the number of vertices of degree $k$.
	Then, the local clustering coefficient of vertices with degree $k$ equals
	\begin{equation}
	\label{c(k)-def}
	c(k)=
	\frac{1}{N_k}\frac{2\triangle_k}{k(k-1)},
	\end{equation}
	for all $k$ with $N_k\geq 1$. Note that $c(k)$ is not defined if $N_k=0$.
	The local clustering coefficient can be interpreted as the probability that two randomly chosen neighbors of a vertex of degree $k$ connect to one another. Typically, $c(k)$ is a decreasing function of $k$.
	
	
	The next theorem shows the behavior of $c(k)$ in uniform random graphs:
	\begin{theorem}[Local clustering.]\label{thm:ck}
		Let $\tau\in (2,3)$ and $\boldsymbol{d}_n$ be a degree sequence on $n$ vertices satisfying Assumption~\ref{ass:degreesck}. Let $G_n$ be a uniformly sampled graph from $\mathcal{G}(\boldsymbol{ d}_n)$. Define $A=\pi/\sin(\pi\tau)>0$ for $\tau\in(2,3)$, $\mu=\Exp{D}$ and let $C$ be the constant in~\eqref{D-tailck} and $c(k)$ the local clustering coefficient of $G_n$. Then, as $n\to\infty$,
		\begin{enumerate}[leftmargin=* , labelsep=1.5cm, 
			align=LeftAlignWithIndent, itemsep=-0.1cm]
			\item [\textup{(Range I.)}]  for $1\ll k=o( n^{(\tau-2)/(\tau-1)})$ where $N_k\ge 1$, 
			\begin{equation}
			\frac{c(k)}{n^{2-\tau}\log(n)}\plim \frac{3-\tau}{\tau-1} \mu^{-\tau} (C(\tau-1))^2A,
			\end{equation}
			\item
			[\textup{(Range II.)}] for $k=\omega(n^{(\tau-2)/(\tau-1)})$ and $k=o( \sqrt{n})$ where $N_k\ge 1$,
			\begin{equation}
			\frac{c(k)}{n^{2-\tau}\log(n/k^2)}\plim \mu^{-\tau} (C(\tau-1))^2 A ,
			\end{equation}
			\item[\textup{(Range III.)}] for $k=\lceil B\sqrt{n}\rceil$ where $N_k\geq 1$ and $B>0$ is a constant,
			\begin{equation}
			\frac{c(k)}{n^{2-\tau}}\plim \mu^{2-2\tau}(C(\tau-1))^2\!\!\int_{0}^{\infty}\int_{0}^{\infty}\!\!\frac{(t_1t_2)^{2-\tau}}{(1+t_1B)(1+t_2B)(\mu^{-1}+t_1t_2)}\dd t_1\dd t_2
			\end{equation}
			\item 
			[\textup{(Range IV.)}] for $k=\omega(\sqrt{n})$ and $k\leq d_{\max}$ where $N_k\ge 1$,
			\begin{equation}
			\frac{c(k)}{n^{5-2\tau}k^{2\tau-6}}\plim \mu^{3-2\tau}(C(\tau-1))^2 A^2.
			\end{equation}
		\end{enumerate} 
	\end{theorem}
	Figure~\ref{fig:ck} illustrates the behavior of $c(k)$ in the uniform random graph. For small values of $k$, $c(k)$ is independent of $k$. Then, a range of slow decay in $k$ follows. When $k\gg\sqrt{n}$, $c(k)$ starts to decay as a power of $k$. A Taylor expansion of the behavior of $c(k)$ in the third range for $B$ small and $B$ large shows that the behavior of the scaling limit of $k\mapsto c(k)$ between Ranges II and IV is smooth (see also~\cite[Theorem~3]{hofstad2017c})
	\begin{figure}
		\centering
		\begin{minipage}{0.45\linewidth}
		\includegraphics[width=\textwidth]{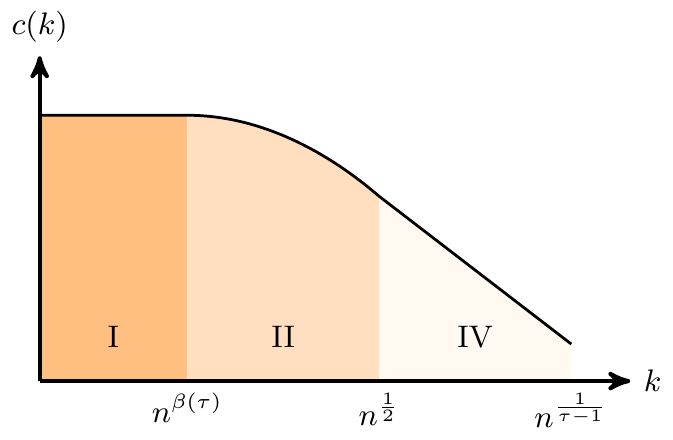}
		\caption{Illustration of the $k\mapsto c(k)$ curve of Theorem~\ref{thm:ck}.}
		\label{fig:ck}
		\end{minipage}
	\hspace{0.2cm}
	\begin{minipage}{0.45\linewidth}
		\includegraphics[width=\textwidth]{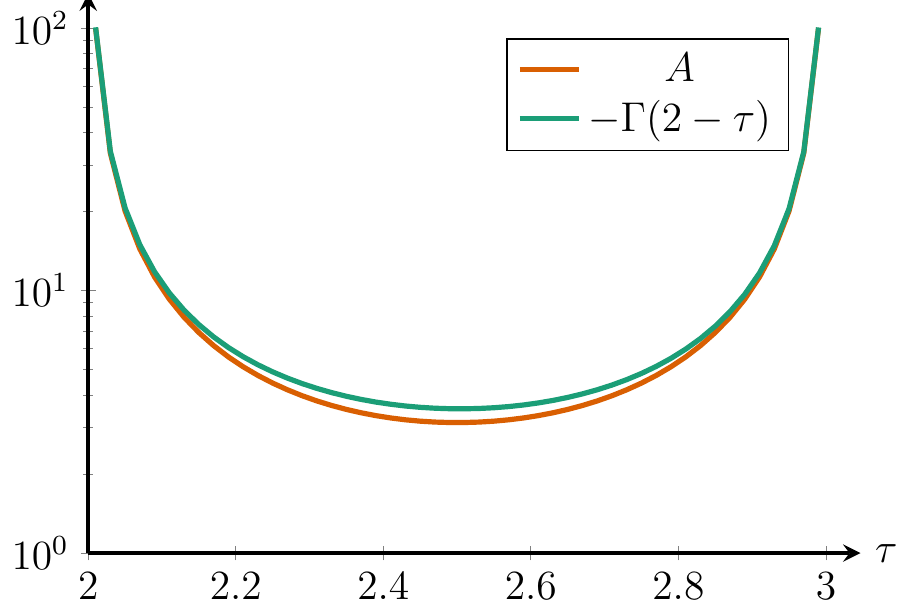}
		\caption{The constant $A=\pi/\sin(\pi \tau)$ of Theorem~\ref{thm:ck} and $-\Gamma(2-\tau)$.}
		\label{fig:constcomp}
	\end{minipage}
	\end{figure}
	
	\emph{Main contributions.}
	As in the case of triangle counts, we show that the constrained triangle counts are dominated by triangle counts between vertices of degrees in specific ranges. Because the triangles are constrained to contain one vertex of degree $k$, these ranges describe the degrees of the two other vertices involved in the triangle. When $k$ is in Range I, the largest contribution to the number of constrained triangles is from vertices $i$ and $j$ such that $d_id_j$ scales as $n$. Note that the most contributing vertices here do not depend on $k$, explaining the independence of $c(k)$ of $k$ in the initial range. When $k$ is in Range II, the largest contribution to the number of constrained triangles is from vertices $i$ and $j$ such that $d_id_j$ scales as $n$ and additionally $d_i,d_j\leq n/k$. This extra constraint causes the mild dependence on $k$ in the second range. In the fourth range, the vertices that contribute most to $c(k)$ satisfy that $d_i$ and $d_j$ scale as $n/k$. In this last regime, the degrees of the other two vertices involved in the triangle clearly depend on $k$. 
	\\~\\
	\emph{Comparison with the erased configuration model.}
	The result for $c(k)$ is also very similar in the erased configuration model and the uniform random graph. In fact, a similar theorem as Theorem~\ref{thm:ck} holds for the erased configuration model, when the constant $A$ is replaced by $-\Gamma(2-\tau)$, where $\Gamma$ denotes the Gamma function. Note that $-\Gamma(2-\tau)>\pi/(\sin(\pi\tau))=A$ for $\tau\in(2,3)$, as Figure~\ref{fig:constcomp} shows. Thus, the local clustering coefficient of the erased configuration model is higher than the local clustering coefficient of a uniform random graph of the same degree sequence, similarly to triangle counts.
	\\~\\
	\emph{Assumption on the degree sequence.}
	To prove Theorem~\ref{thm:triang}, we assume a pure power-law degree distribution for vertices with degrees at most $\sqrt{n}$ (Assumption~\ref{ass:degrees}), since these vertices form the main contribution to the number of triangles. For Theorem~\ref{thm:ck}, we assume a pure power-law degree distribution all the way up to $n^{1/(\tau-1)}/\log(n)$. However, this strong assumption is only necessary in Range I, because there the most contributing triangles containing a vertex of degree $k$  may have degrees as high as $n^{1/(\tau-1)}/\log(n)$. In Ranges II, III and IV, the main contributing triangles with a vertex of degree $k$ have degrees at most $n/k$, so that it suffices to assume the pure power law~\eqref{D-tailck} only for $j=O(n/k)$. However, for ease of notation we use Assumption~\ref{ass:degreesck}\ref{ass:degreerangeck} throughout the rest of this paper.
		
	\section{Connection probability estimates}\label{sec:conprob}
	In this section, we estimate the connection probability between vertices of specific degrees in a uniform random graph, which is the key ingredient for proving Theorems~\ref{thm:triang} and~\ref{thm:ck}.
Recall that $L_n=\sum_{i=1}^nd_i$. 
Furthermore, let $\{u\sim v\}$ denote the event that vertex $u$ is connected to $v$.

Our key lemma is as follows, where the probability space refers to the uniformly random simple graphs with degree sequence $\boldsymbol{d}$:

\begin{lemma}\label{lem:edgeProb}
	Assume that $\tau\in(2,3)$ is fixed and
	$\boldsymbol{d}$ satisfies Assumption~\ref{ass:degrees}\ref{ass:degreeall} with $\tau\in(2,3)$. Assume further that $L_n/n \nrightarrow 0$. Let $U$ denote a set of unordered pairs of vertices and let $E_U$ denote the event that $\{x,y\}$ is an edge for every $\{x,y\}\in U$. Then, assuming that $|U|=O(1)$, for every $\{u,v\}\notin U$,
	\begin{equation}\label{eq:pijon}
	\Prob{u\sim v\mid E_U} = (1+o(1)) \frac{(d_u-|U_u|)(d_v-|U_v|)}{L_n+(d_u-|U_u|)(d_v-|U_v|)},
	\end{equation}
	where $U_x$ denote the set of pairs in $U$ that contain $x$.
\end{lemma}

\begin{remark}\label{rem:edgeprob}
	Lemma~\ref{lem:edgeProb} shows that when $d_ud_v\gg L_n$, then
	\begin{equation}
		1-\Prob{u\sim v}=(1+o(1))\frac{L_n}{d_ud_v}.
	\end{equation}
	In the erased configuration model on the other hand~\cite{hofstad2005},
	\begin{equation}
	1-\Prob{u\sim v}\leq \me^{-d_ud_v/L_n}.
	\end{equation}
	Thus, the probability that two high-degree vertices are not connected  
	decreases at an exponential rate in $L_n/d_ud_v$, whereas this rate is polynomial in the uniform random graph model. Thus, even though the results on clustering are similar in the 
	two models, there are edge probabilities that behave significantly differently.
	\end{remark}

%

We now proceed to prove Lemma~\ref{lem:edgeProb}. 
As a preparation, we first prove a lemma about the number of 2-paths starting from a specified vertex. 
\begin{lemma}\label{lem:2-paths} 
	Assume that
	$\boldsymbol{d}$ satisfies Assumption~\ref{ass:degrees}\ref{ass:degreeall} with fixed exponent $\tau\in(2,3)$. For any graph $G$ whose degree sequence is $\boldsymbol{d}$, 
	the number of 2-paths starting from any specified vertex is $o(n)$. 
\end{lemma}
\begin{proof} Without loss of generality we may assume that $d_1 \geq d_2 \geq \dots \geq d_n$. For every $1\le i\le n$, the number of vertices with degree at least $d_i$ is at least $i$. By Assumption~\ref{ass:degrees}\ref{ass:degreeall}, we then have $Knd_i^{1-\tau}\ge i$ for every $1\le i\le n$. It follows then that $d_i \leq  \left(K n/i\right)^{\frac{1}{\tau -1}} $. Then the number of 2-paths from any specified vertex is bounded by $\sum_{i=1}^{d_1} d_i$, which is at most
	\begin{align*}
	\sum_{i\geq 1}^{d_1}\left( \frac{Kn}{i} \right)^{{1}/{(\tau-1)}} = \left(Kn \right)^{{1}/{(\tau-1)}} \sum_{i=1}^{d_1} i^{-{1}/{(\tau - 1)}}=O\left(n^{{1}/{(\tau-1)}}\right) d_1^{\frac{\tau-2}{\tau - 1}}= O\left( n^{\frac{2\tau-3}{(\tau-1)^2}} \right),
	\end{align*}
	since $d_1 \le (Kn)^{{1}/{(\tau-1)}}$. Since $\tau\in(2,3) $ the above is $o(n)$. 
	
\end{proof}

\begin{proof}[Proof of Lemma~\ref{lem:edgeProb}.]
To estimate $\CProb{u\sim v}{E_U}$, we will switch between two classes of graphs $S$ and $\bar{S}$. $S$ consists of graphs where all edges in $\{u,v\}\cup U$ are present, whereas  $\bar{S}$ consists of all graphs where every $\{x,y\}\in U$ represents an edge whereas $\{u,v\}$ is not an edge.  
Note that 
\begin{equation}
\label{eq:pxuvS}
\CProb{u\sim v}{E_U} = \frac{|S|}{|S|+|\bar S|} = \frac{1}{1+{|\bar S|}/{|S|}}.
\end{equation}
In order to estimate the ratio $|\bar S|/|S|$, we will define an operation called a {\em forward switching} which converts a graph in $G\in S$ to a graph $G'\in \bar S$. The reverse operation converting $G'$ to $G$ is called a {\em backward switching}. Then we estimate $|\bar S|/|S|$ by counting the number of forward switchings that can be applied to a graph $G\in S$, and the number of backward switchings that can be applied to a graph $G'\in \bar S$.
\begin{figure}
	\begin{center}
		\begin{tikzpicture}[line width=.5pt,vertex/.style={circle,inner sep=0pt,minimum size=0.2cm}, scale=.6]
		

		\node [black, fill=black,label={[label distance=0mm]0:$y$}] (y) at (0:2)  [vertex]{}; 
		\node [black, fill=red,label={[label distance=0mm]60:$v$}] (v) at (60:2)  [vertex]{}; 
		\node [black, fill=red,label={[label distance=0mm]120:$u$}] (u) at (120:2)  [vertex]{};
		\node [black, fill=black,label={[label distance=0mm]180:$x$}] (x) at (180:2)  [vertex]{}; 
		\node [black, fill=black,label={[label distance=0.5mm]240:$a$}] (a) at (240:2) [vertex]{};  
		\node [black, fill=black,label={[label distance=0mm]300:$b$}] (b) at (300:2)  [vertex]{}; 
		
		
		\draw[thick, red,-] (u) to (v);
		\draw[thick, -] (x) to (a);
		\draw[thick, -] (y) to (b);
		\draw[thick,dashed] (u) to (x);
		\draw[thick, dashed] (b) to (a);
		\draw[thick, dashed] (y) to (v);
		
		\draw[thick, <->] (3,0) to (4,0);

		\end{tikzpicture}
		\begin{tikzpicture}[line width=.5pt,vertex/.style={circle,inner sep=0pt,minimum size=0.2cm}, scale=.6]
		

		\node [black, fill=black,label={[label distance=0mm]0:$y$}] (y) at (0:2)  [vertex]{}; 
		\node [black, fill=red,label={[label distance=0mm]60:$v$}] (v) at (60:2)  [vertex]{}; 
		\node [black, fill=red,label={[label distance=0mm]120:$u$}] (u) at (120:2)  [vertex]{};
		\node [black, fill=black,label={[label distance=0mm]180:$x$}] (x) at (180:2)  [vertex]{}; 
		\node [black, fill=black,label={[label distance=0.5mm]240:$a$}] (a) at (240:2) [vertex]{};  
		\node [black, fill=black,label={[label distance=0mm]300:$b$}] (b) at (300:2)  [vertex]{}; 
		
		
		\draw[thick, red, dashed] (u) to (v);
		\draw[thick,-] (u) to (x);
		\draw[thick, -] (b) to (a);
		\draw[thick, -] (y) to (v);
		\draw[thick, dashed] (x) to (a);
		\draw[thick, dashed] (y) to (b);

		\end{tikzpicture}
		
	\end{center}
	\caption{Forward and backward switchings}
	\label{fig:switching}
\end{figure}
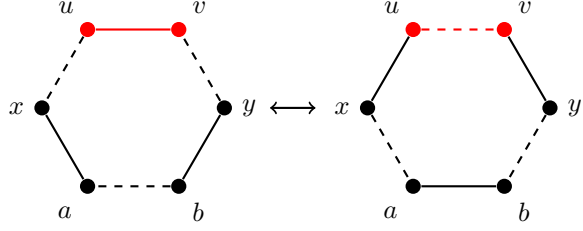

	The forward switching is defined by choosing two edges and specifying their ends as $\{x,a\}$ and $\{y,b\}$. The choice must satisfy the following constraints:
\begin{enumerate}
	\item None of $\{u,x\}$, $\{v,y\}$, or $\{a,b\}$ is an edge;
	\item $\{x,a\},\{y,b\}\notin U$;
	\item All of $u$, $v$, $x$, $y$, $a$, and $b$ must be distinct except that $x=y$ is permitted.
\end{enumerate} 
Given a valid choice, the forward switching replaces  the three edges $\{u,v\}$, $\{x,a\}$, and $\{y,b\}$ by $\{u,x\}$, $\{v,y\}$, and $\{a,b\}$. Note that the forward switching preserves the degree sequence, and converts a graph in $S$ to a graph in $\bar S$. The inverse operation of a forward switching is called a backward switching. See Figure~\ref{fig:switching} for an illustration.

Next, we estimate the number of ways to perform a forward switching to a graph $G$ in $S$, denoted by $f(G)$, and the number of ways to perform a backward switching to a graph $G'$ in $\bar S$, denoted by $b(G)$. The number of total switchings between $S$ and $\bar S$ is equal to $|S| \Exp{f(G)}=|\bar S| \Exp{b(G')}$, where the expectation is over a uniformly random $G\in S$ and $G'\in \bar S$ respectively. Consequently,
\begin{equation}\label{eq:SSprime}
\frac{|\bar S|}{|S|} = \frac{\Exp{ f(G)}}{\Exp{ b(G')}}.
\end{equation}

Given an arbitrary graph $G\in S$, the number of ways of carrying out a forward switching is at most $L_n^2$, since there are at most $L_n$ ways to choose $\{x,a\}$, and at most $L_n$ ways to choose $\{y,b\}$. Note that choosing $\{x,a\}$ for the first edge and $\{y,b\}$ for the second edge results in a different switching than vice versa. To find a lower bound on the number of ways of performing a forward switching, we subtract from $L_n^2$ an upper bound on the number of invalid choices for $\{x,a\}$ and $\{y,b\}$. These can be summarized as follows:
\begin{enumerate}[label=(\alph*)]
	\item At least one of $\{u,x\}, \{a,b\}, \{v,y\}$ is an edge,
	\item At least one of $\{x,a\}$ or $\{y,b\}$ is in $U$,
	\item Any vertex overlap other than $x=y$ (i.e. if one of $a$ or $b$ is equal to one of $x$ or $y$, or if $a=b$, or if one of $u$ or $v$ are one of $\left\{a,b,x,y \right\}$).
\end{enumerate}
To find an upper bound for (a), note that any choice in case (a) must involve a single edge, and a 2-path starting from a specified vertex. By Lemma~\ref{lem:2-paths}, the number of choices for (a) then is upper bounded by $3\cdot o(L_n) \cdot L_n = o(L_n^2)$. The number of choices for case (b) is $O(L_n)$ as $|U|=O(1)$, and there are at most $L_n$ ways to choose the other edge which is not restricted to be in $U$. To bound the number of choices for (c), we investigate each case:
\begin{enumerate}
	\item[(C1)] $a$ or $b$ is equal to $x$ or $y$; or $a=b$. In this case, $x,y,a,b$ forms a 2-path. Thus, there are at most $5\cdot n \cdot o(L_n)=o(L_n^2)$ choices (noting that $n=O(L_n)$), where $n$ is the number of ways to choose a vertex, and $o(L_n)$ bounds the number of 2-paths starting from this specified vertex;
	\item[(C2)] one of $u$ and $v$ is one of $\left\{a,b,x,y \right\}$. In this case, there is one 2-path starting from $u$ or $v$, and a single edge. Thus, there are at most $8\cdot L_n d_{\max} =o(L_n^2)$ choices, where  $d_{\max}$ bounds the number of ways to choose a vertex adjacent to $u$ or $v$ and $L_n$ bounds the number of ways to choose a single edge.
\end{enumerate}
Thus, the number of invalid choices for  $\{x,a\}$ and $\{y,b\}$ is $o(L_n^2)$, so that the number of forward switchings which can be applied to any $G\in S$ is $(1+o(1))L_n^2$. Thus,
\begin{equation}\label{eq:expfG}
\Exp{f(G)}=L_n^2(1+o(1)).
\end{equation}

Given a graph $G'\in\bar S$, consider the backward switchings that can be applied to $G'$. There are at most $L_n(d_u-|U_u|)(d_v-|U_v|)$ ways to do the backward switching, since we are choosing an edge which is adjacent to $u$ but not in $U$, an edge which is adjacent to $v$ but not in $U$, and another ``oriented'' edge $\{a,b\}$ (oriented in the sense that each edge has two ways to specify its end vertices as $a$ and $b$). For a lower bound, we consider the following forbidden choices:
\begin{enumerate}[label=(\alph*)]
	\item[(a$'$)] at least one of $\{x,a\}$ or $\{y,b\}$ is an edge,
	\item[(b$'$)] $\{a,b\}\in U$,
	\item[(c$'$)] any vertices overlap other than $x=y$ (i.e. if $\{a,b\}\cap \{u,v,x,y\}\neq \varnothing$).
\end{enumerate}
For (a$'$), suppose that $\{x,a\}$ is present, giving the two-path $\{x,a\},\{a,b\}$ in $G' $. There are at most $(d_u-|U_u|)(d_v-|U_v|)$ ways to choose $x$ and $y$. Given any choice for $x$ and $y$, there are at most $o(L_n)$ ways to choose a 2-path starting from $x$, and hence $o(L_n)$ ways to choose $a,b$. Thus, the total number of choices is at most $o((d_u-|U_u|)(d_v-|U_v|)L_n)$. The case that $\{y,b\}$ is an edge is symmetric. 

For (b$'$), there are $O(1)$ choices for choosing $\{a,b\}$ since $|U|=O(1)$, and at most $(d_u-|U_u|)(d_v-|U_v|)$ choices $x$ and $y$. Thus, the number of choices for case (b$'$) is $O((d_u-|U_u|)(d_v-|U_v|))=o((d_u-|U_u|)(d_v-|U_v|)L_n)$. 

For (c$'$), the case that $a$ or $b$ is equal to $x$ or $y$ corresponds to a 2-path starting from $u$ or $v$ together with a single edge from $u$ or $v$. Since $o(L_n)$ bounds the number of 2-paths starting from $u$ or $v$ and $d_u-|U_u|+d_v-|U_v|$ bounds the number of ways to choose the single edge, there are $o(L_n(d_v-|U_v|)) + o(L_n(d_u-|U_u|))$ total choices. If $a$ or $b$ is equal to $u$ or $v$, there are $(d_u-|U_u|)(d_v-|U_v|)$ ways to choose $x$ and $y$, and at most $d_u+d_v$ ways to choose the last vertex as a neighbor of $u$ or $v$. Thus, there are $O((d_u-|U_u|)(d_v-|U_v|) d_{\max})=o((d_u-|U_u|)(d_v-|U_v|)L_n)$ total choices, since $d_{\max} = O( n^{{1}/{(\tau-1)}} ) = o(n) = o(L_n)$. This concludes that the number of backward switchings that can be applied to any graph $G'\in S'$ is $(d_u-|U_u|)(d_v-|U_v|)L_n(1+o(1))$, so that also
\begin{equation}\label{eq:bG}
\Exp{b(G')}=(d_u-|U_u|)(d_v-|U_v|)L_n(1+o(1)).
\end{equation}

Combining~\eqref{eq:SSprime},~\eqref{eq:expfG} and~\eqref{eq:bG} results in
\[
{|\bar S|}/{|S|} = (1+o(1))\frac{L_n^2}{(d_u-|U_u|)(d_v-|U_v|)L_n},
\] 
and thus~\eqref{eq:pxuvS} yields
\[
\Prob{u\sim v \mid E_U} = \frac{1}{1+{|\bar{S}|}/{|S|}} = (1+o(1))\frac{(d_u-|U_u|)(d_v-|U_v|)}{L_n+(d_u-|U_u|)(d_v-|U_v|)}. 
\]
\end{proof}

	\section{Proof of Theorem~\ref{thm:triang}}\label{sec:prooftriang}

	In this section, we use Lemma~\ref{lem:edgeProb} to prove Theorem~\ref{thm:triang}. 
	Let
	\begin{equation}
B_n(\varepsilon)=\left\{u:d_u\in [\varepsilon\sqrt{\mu n},\sqrt{\mu n}/\varepsilon]\right\}
\end{equation} 
for some fixed $\varepsilon>0$. 
Define $X_{uvw}=1_{\{u\sim v, v\sim w, u\sim w\}}$ and
\[
\triangle_G(B_n(\varepsilon))=\sum_{u,v,w\in B_n(\varepsilon)} X_{uvw}, \quad \triangle_G(\bar{B}_n(\varepsilon)) = \sum_{u\notin B_n(\varepsilon)} \sum_{v,w\in[n]} X_{uvw}.
\]
Thus,  $\triangle_G(B_n(\varepsilon))$ denotes the number of triangles in a graph $G$ where all degrees of the vertices in the triangle are in $B_n(\varepsilon)$ and $\triangle_{G}(\bar{B}_n(\varepsilon))$ denotes the number of triangles where at least one of the vertices does not have its degree in $B_n(\varepsilon)$. 
Then, the following lemma bounds $\triangle_G(\bar{B}_n(\varepsilon))$:
	
	\begin{lemma}\label{lem:boundcontr}
	Let $\tau\in (2,3)$ and $\boldsymbol{d}_n$ be a degree sequence on $n$ vertices satisfying Assumption~\ref{ass:degrees}. Let $G_n$ be a uniformly sampled graph from $\mathcal{G}(\boldsymbol{ d}_n)$. Then,
		\begin{equation}
			\lim_{\varepsilon\downarrow 0}\limsup_{n\to\infty}\frac{\Exp{\triangle_{G_n}(\bar{B}_n(\varepsilon))}}{n^{\frac{3}{2}(3-\tau)}}=0. 
					\end{equation}
	\end{lemma}
	\begin{proof}
		Let $F=[0,\infty)^3\setminus[\varepsilon\sqrt{\mu n},\sqrt{\mu n}/\varepsilon]^3$, so that $F$ denotes the area where one of the three degrees is not in $B_n(\varepsilon)$. By~\eqref{eq:bound}, for some $K_1>0$
		\begin{equation}\label{eq:expBbar}
			\begin{aligned}[b]
			&\Exp{\triangle_{G_n}(\bar{B}_n(\varepsilon))} =\sum_{1\leq u<v<w\leq n}\Prob{u\sim v, v\sim w, u\sim w}\ind{u,v \text{ or }w\notin B_n(\varepsilon)}\\
			&\leq K_1\sum_{1\leq u<v<w\leq n}\min\left(\frac{d_ud_v}{L_n},1\right)\min\left(\frac{d_ud_w}{L_n},1\right)\min\left(\frac{d_vd_w}{L_n},1\right)\ind{u,v \text{ or }w\notin B_n(\varepsilon)}\\
			& = K_1n^3\int\int\int_F\min\left(\frac{xy}{L_n},1\right)\min\left(\frac{xz}{L_n},1\right)\min\left(\frac{yz}{L_n},1\right)\dd F_n(x)\dd F_n(y)\dd F_n(z),
			\end{aligned}
		\end{equation}
		where the factor $n^3$ arises from replacing the sum over all sets of three vertices by the average over three uniformly chosen vertices.
		
		
		We first investigate the contribution to the integral from the area where the first entry is in $[0,\varepsilon\sqrt{\mu n}]$. 
		We use that for a bounded, absolutely continuous function $g(x)$ such that $g(0)=0$,
		\begin{equation}\label{eq:gxint}
		\begin{aligned}[b]
			\int_{0}^\infty g(x)\dd F_n(x)& = \int_{0}^\infty g'(x)(1-F_n(x))\dd x \leq K\int_{0}^{\infty}g'(x)x^{1-\tau}\dd x\\
			& =K\int_{0}^{\infty}(\tau-1)g(x)x^{-\tau}\dd x +K\big[ g(x)x^{1-\tau}\big]_0^\infty \\
			& = K(\tau-1)\int_{0}^{\infty}g(x)x^{-\tau}\dd x,
			\end{aligned}
		\end{equation} 
		where $g'(x)$ denotes a function such that $\int_0^xg'(y)\dd y=g(x)$, and where we used Assumption~\ref{ass:degrees}\ref{ass:degreeall}.
		Thus, for some $K_2>0$, the integral in~\eqref{eq:expBbar} can be bounded by
			\begin{align}
			& K_2\int_0^{\varepsilon\sqrt{\mu n}}\int_0^\infty\int_0^\infty (xy)^{-\tau}\min\left(\frac{xy}{L_n},1\right)\min\left(\frac{xz}{L_n},1\right)\min\left(\frac{yz}{L_n},1\right)\dd x\dd y\dd F_n(z) \\
			\end{align}

			
			Furthermore, for all non-decreasing $g$ that are bounded on $[0,\varepsilon \sqrt{\mu n}]$ and once differentiable,
			\begin{align}
			\int_{0}^{\varepsilon\sqrt{\mu n}}g(x)\dd F_n(x)&= \int_{0}^{\varepsilon\sqrt{\mu n}}\int_0^{x}g'(y)\dd y\dd F_n(x)\nonumber\\
			& =  \int_{0}^{\varepsilon\sqrt{\mu n}}(F_n(\varepsilon\sqrt{\mu n})-F_n(y))g'(y)\dd y\nonumber\\
			& \leq   K\int_{0}^{\varepsilon\sqrt{\mu n}}y^{1-\tau}g'(y)\dd y\nonumber\\
			& =K(\tau-1) \int_{0}^{\varepsilon\sqrt{\mu n}}y^{-\tau}g(y)\dd y +g(y)y^{1-\tau}\big]_0^{\varepsilon\sqrt{\mu n}}\nonumber\\
			& \leq K(\tau-1) \int_{0}^{\varepsilon\sqrt{\mu n}}y^{-\tau}g(y)\dd y +g(\varepsilon\sqrt{\mu n})(\varepsilon\sqrt{\mu n})^{1-\tau}
			\end{align}
			In this case we take
			\begin{equation}\label{eq:gn}
			g_n(z)=g(z)=\int_0^\infty\int_0^\infty (xy)^{-\tau}\min\left(\frac{xy}{L_n},1\right)\min\left(\frac{xz}{L_n},1\right)\min\left(\frac{yz}{L_n},1\right)\dd x\dd y 
			\end{equation}
			Thus, using that $g(\varepsilon\sqrt{\mu n}) = O(n^{1-\tau}\varepsilon^{2\tau-2})$,
				\begin{align}
			& \int_0^{\varepsilon\sqrt{\mu n}}\int_0^\infty\int_0^\infty (xy)^{-\tau}\min\left(\frac{xy}{L_n},1\right)\min\left(\frac{xz}{L_n},1\right)\min\left(\frac{yz}{L_n},1\right)\dd x\dd y\dd F_n(z) \nonumber\\
			& \leq K_3 \int_0^{\varepsilon\sqrt{\mu n}}\int_0^\infty\int_0^\infty (xyz)^{-\tau}\min\left(\frac{xy}{L_n},1\right)\min\left(\frac{xz}{L_n},1\right)\min\left(\frac{yz}{L_n},1\right)\dd x\dd y\dd z\nonumber\\
			& \quad +O(n^{\frac{3}{2}(1-\tau)})\varepsilon^{\tau-1}\nonumber\\
			& = h(\varepsilon)\bigO{n^{\frac{3}{2}(1-\tau)}},
			\end{align}
			for some function $h(\varepsilon)$ such that $\lim_{\varepsilon\downarrow 0}h(\varepsilon)=0$ and some $K_3>0$. The last equality follows from~\cite[Lemma 4.2]{hofstad2017}.  Multiplying by $n^3$ as in~\eqref{eq:expBbar} then proves that the contribution to~\eqref{eq:expBbar} where at least one of the vertices has degree smaller than $\varepsilon\sqrt{n}$ is small.

			We next investigate the contribution to~\eqref{eq:expBbar} where the first entry is in $[\sqrt{\mu n}/\varepsilon,\infty)$. Again using~\eqref{eq:gxint}, we can write for some $K_2>0$
			\begin{align}\label{eq:intFnhigh}
			& \int_{\sqrt{\mu n}/\varepsilon}^\infty\int_0^\infty\int_0^\infty \min\left(\frac{xy}{L_n},1\right)\min\left(\frac{xz}{L_n},1\right)\min\left(\frac{yz}{L_n},1\right)\dd F_n(x)\dd F_n(y)\dd F_n(z)\nonumber\\
			&\leq K_2 \int_{\sqrt{\mu n}/\varepsilon}^\infty\int_0^\infty\int_0^\infty (xy)^{-\tau}\min\left(\frac{xy}{L_n},1\right)\min\left(\frac{xz}{L_n},1\right)\min\left(\frac{yz}{L_n},1\right)\dd x\dd y\dd F_n(z)\nonumber\\
			&= K_2\Exp{g_n(D_n)\ind{D_n>\sqrt{\mu n}/\varepsilon}},
			\end{align}
			where $g_n$ is as in~\eqref{eq:gn} and $D_n\sim F_n$. Since $g_n$ is non-decreasing, we get 
			\begin{equation}
			\Exp{g_n(D_n)\ind{D_n>\sqrt{\mu n}/\varepsilon}}\leq \Exp{g_n(\bar{D})\ind{\bar{D}>\sqrt{\mu n}/\varepsilon}},
			\end{equation}
			where $\Prob{\bar{D}>x}=\min(Kx^{1-\tau},1$, as in Assumption~\ref{ass:degreesck}\ref{ass:degreerangeck}. 
	Thus, the contribution in~\eqref{eq:intFnhigh} can be bounded by
	\begin{align}
	 & \int_{\sqrt{\mu n}/\varepsilon}^\infty\int_0^\infty\int_0^\infty (xyz)^{-\tau}\min\left(\frac{xy}{L_n},1\right)\min\left(\frac{xz}{L_n},1\right)\min\left(\frac{yz}{L_n},1\right)\dd x\dd y\dd z +\varepsilon^{\tau-1}O(n^{\frac{3}{2}(1-\tau)})\nonumber \\
	 & =h(\varepsilon)O(n^{\frac{3}{2}(1-\tau)})+\varepsilon^{\tau-1}O(n^{\frac{3}{2}(1-\tau)}),
	\end{align}
	for some function $h(\varepsilon)$ such that $\lim_{\varepsilon\downarrow 0}h(\varepsilon)=0$, where the last equality follows from~\cite[Lemma 4.2]{hofstad2017}. Again, multiplying by $n^3$ then shows that the contribution to~\eqref{eq:expBbar} from the situation where at least one of the vertices has degree larger than $\sqrt{n}/\varepsilon$ is sufficiently small.			
	\end{proof}

We now investigate the expected value of $\triangle_{G_n}(B_n(\varepsilon))$:
\begin{lemma}\label{lem:expcond}
	Let $\tau\in (2,3)$ and $\boldsymbol{d}_n$ be a degree sequence on $n$ vertices satisfying Assumption~\ref{ass:degrees}. Let $G_n$ be a uniformly sampled graph from $\mathcal{G}(\boldsymbol{ d}_n)$. Then,
	\begin{equation}
	\frac{\Exp{\triangle_{G_n}(B_n(\varepsilon))}}{n^{\frac{3}{2}(3-\tau)}\mu^{-\frac{3}{2}(\tau-1)}}\to \frac{1}{6}(C(\tau-1))^3\int_{\varepsilon}^{1/\varepsilon}\int_{\varepsilon}^{1/\varepsilon}\int_{\varepsilon}^{1/\varepsilon}(xyz)^{-\tau}\frac{xy}{1+xy}\frac{xz}{1+xz}\frac{yz}{1+yz}\dd x\dd y \dd z.
	\end{equation}
\end{lemma}
\begin{proof}
	Let $\triangle_{i,j,k}$ denote the event that a triangle is present between vertices $i,j$ and $k$. Then,
	\begin{equation}\label{eq:probtriang}
		\Prob{\triangle_{i,j,k}}=\Prob{i\sim j}\Prob{j\sim k\mid i\sim j}\Prob{i\sim k \mid i\sim j, j\sim k}.
	\end{equation}
	Let 
	\begin{equation}\label{eq:gfunc}
		g(d_i,d_j,d_k)=\frac{d_id_j}{d_id_j+\mu n}\frac{d_id_k}{d_id_k+\mu n}\frac{d_jd_k}{d_jd_k+\mu n}.
	\end{equation}
	When $i,j,k\in B_n(\varepsilon)$, $d_i,d_j,d_k\gg 1$ and also $d_id_j,d_id_k,d_jd_k=O(n)$, so that we can use~\eqref{eq:pijon} in Lemma~\ref{lem:edgeProb} to conclude
	\begin{equation}\label{eq:pijk}
		\Prob{\triangle_{i,j,k}}=\frac{d_id_j}{d_id_j+L_n}\frac{(d_i-1)d_k}{d_id_k+L_n}\frac{(d_j-1)(d_k-1)}{d_jd_k+L_n}(1+o(1))=g(d_i,d_j,d_j)(1+o(1)).
	\end{equation}
	We then use that
	\begin{align}\label{eq:expcondg}
	\Exp{\triangle_{G_n}(B_n(\varepsilon))} & = \frac{1}{6}\sideset{}{'}\sum_{i,j,k}\Prob{\triangle_{i,j,k}}\ind{i,j,k\in B_n(\varepsilon)} \nonumber\\
	& =(1+o(1)) \frac{1}{6}\sideset{}{'}\sum_{i,j,k}g(d_i,d_j,d_k)\ind{i,j,k\in B_n(\varepsilon)},
	\end{align}
	where $\sideset{}{'}\sum$ denotes the sum over distinct indices.
	
	We then define the measure
	\begin{equation}
	\Mn([a,b])=\mu^{(\tau-1)/2}n^{(\tau-3)/2}\sum_{i\in[n]}\ind{d_i\in[a,b]\sqrt{\mu n}}.
	\end{equation}
	By Assumption~\ref{ass:degrees}\ref{ass:degreerange}, 
	\begin{align}
	n^{-1}\sum_{i\in[n]}\ind{d_i\in[a,b]\sqrt{\mu n}}&=C\sqrt{\mu n}^{1-\tau}(a^{1-\tau}-b^{1-\tau})(1+o(1))\nonumber\\
	& =C(\tau-1)\sqrt{\mu n}^{1-\tau}\int_{a}^{b}t^{-\tau}\dd t(1+o(1)),
	\end{align}
	 so that
 	\begin{equation}\label{eq:measconv}
 	\Mn([a,b])\to C(\tau-1)\int_{a}^{b}t^{-\tau}\dd t =:\lambda([a,b]).
 	\end{equation}
 	Furthermore, we can write 
 	\begin{align}
 	& \frac{\sum_{1\leq i<j<k\leq n}g(d_i,d_j,d_k)\ind{i,j,k\in B_n(\varepsilon)}}{n^{\frac{3}{2}(3-\tau)}\mu^{-\frac{3}{2}(\tau-1)}}\nonumber\\
 	& =\frac{1}{6}\int_{\varepsilon}^{1/\varepsilon}\int_{\varepsilon}^{1/\varepsilon}\int_{\varepsilon}^{1/\varepsilon} g(t_1,t_2,t_3)\dd \Mn(t_1)\dd \Mn(t_2)\dd \Mn(t_3)
 	\end{align}
	Because the function $g(t_1,t_2,t_3)$ is a bounded, continuous function on $[\varepsilon,1/\varepsilon]^3$, 
	\begin{equation}\label{eq:gconv}
	\begin{aligned}[b]
	& \frac{\sum_{1\leq i<j<k\leq n}g(d_i,d_j,d_k)\ind{i,j,k\in B_n(\varepsilon)}}{n^{\frac{3}{2}(3-\tau)}\mu^{-\frac{3}{2}(\tau-1)}} \to\!\! \frac{1}{6}\int_{\varepsilon}^{1/\varepsilon}\int_{\varepsilon}^{1/\varepsilon}\int_{\varepsilon}^{1/\varepsilon}\!\! g(t_1,t_2,t_3)\dd \lambda(t_1)\dd \lambda(t_2)\dd \lambda(t_3)\\
	&=\frac{(C(\tau-1))^3}{6}\int_{\varepsilon}^{1/\varepsilon}\int_{\varepsilon}^{1/\varepsilon}\int_{\varepsilon}^{1/\varepsilon}(xyz)^{-\tau}\frac{xy}{1+xy}\frac{xz}{1+xz}\frac{yz}{1+yz}\dd x\dd y \dd z.
	\end{aligned}
	\end{equation}
	Combining this with~\eqref{eq:expcondg} proves the lemma.
\end{proof}

We proceed to investigate the variance of $\triangle(B_n(\varepsilon))$. The next lemma shows that the variance of $\triangle(B_n(\varepsilon))$ is with high probability small compared to the square of its expectation. 

\begin{lemma}\label{lem:vartriang}
	Let $\tau\in (2,3)$ and $\boldsymbol{d}_n$ be a degree sequence on $n$ vertices satisfying Assumption~\ref{ass:degrees}. Let $G_n$ be a uniformly sampled graph from $\mathcal{G}(\boldsymbol{ d}_n)$. Then,
	\begin{equation}
		\frac{\Var{\triangle_{G_n}(B_n(\varepsilon))}}{\Exp{\triangle_{G_n}(B_n(\varepsilon))}^2}\to 0
	\end{equation}
\end{lemma}
\begin{proof}
	By Lemma~\ref{lem:expcond}, $\Exp{\triangle_{G_n}(B_n(\varepsilon))}=\Theta(n^{\frac{3}{2}(3-\tau)})$. Thus, we need to show that the variance scales as $\op(n^{9-3\tau})$.
	We write
	\begin{equation}
		\Var{\triangle_{G_n}(B_n(\varepsilon))}\!=\!\sideset{}{'}\sum_{i,j,k}\sideset{}{'}\sum_{u,v,w}\!\!\left(\Prob{\triangle_{i,j,k}\triangle_{u,v,w}}-\Prob{\triangle_{i,j,k}}\Prob{\triangle_{u,v,w}}\right) \ind{i,j,k,u,v,w\in B_n(\varepsilon)},
	\end{equation}
	where $\sideset{}{'}\sum$ denotes a sum over distinct indices.
	The value of the summand depends on the overlap of the indices $i,j,k$ and $u,v,w$. We denote the contribution to the variance where $r$ indices are distinct by $V^{\sss{(r)}}$. 
	When all 6 indices are different, we obtain similarly to~\eqref{eq:pijk} that
	\begin{equation}
	\Prob{\triangle_{i,j,k}\triangle_{u,v,w}}=g(d_i,d_j,d_k)g(d_u,d_v,d_w)(1+o(1)).
	\end{equation}
	Thus, the contribution to the variance when all 6 indices are different can be bounded by
	\begin{align}
	 V^{\sss{(6)}}
	 & =  \sideset{}{'}\sum_{i,j,k}\sideset{}{'}\sum_{u,v,w}o(g(d_i,d_j,d_k)g(d_u,d_v,d_w))\ind{i,j,k,u,v,w\in B_n(\varepsilon)}\nonumber\\
	 & =o(\Exp{\triangle_{G_n}(B_n(\varepsilon))}^2).
	\end{align}
	When $i=u$, but all other indices are different, we bound the contribution to the variance using Assumption~\ref{ass:degrees}\ref{ass:degreeall} as
	\begin{equation}
	 V^{\sss{(5)}}\leq \sideset{}{'}\sum_{i,j,k,v,w}\ind{i,j,k,v,w\in B_n(\varepsilon)} \leq \bigg(\sum_{i\in[n]}\ind{d_i>\varepsilon\sqrt{n}}\bigg)^5 \leq K^5\varepsilon^{5-5\tau}n^{\frac{5}{2}(3-\tau)}= o(n^{9-3\tau}),
	\end{equation}
	and the other contributions can be bounded similarly.
\end{proof}

We now prove Theorem~\ref{thm:triang} using the above lemmas:
\begin{proof}[Proof of Theorem~\ref{thm:triang}]
	Fix $\varepsilon>0$.
	Applying the Markov inequality together with Lemma~\ref{lem:boundcontr} yields that for all $\delta>0$
	\begin{equation}
		\Prob{\triangle_{G_n}(\bar{B}_n(\varepsilon))>\delta n^{\frac{3}{2}(3-\tau)}}=\bigO{\frac{h(\varepsilon)}{\delta}},
	\end{equation}
	for some $h(\varepsilon)\to 0$ as $\varepsilon\downarrow 0$ so that
	\begin{equation}
		\triangle_{G_n}(\bar{B}_n(\varepsilon))=h(\varepsilon)\bigOp{n^{\frac{3}{2}(3-\tau)}}.
	\end{equation}
	We now focus on $\triangle_{G_n}(B_n(\varepsilon))$. Using Lemma~\ref{lem:vartriang} together with the Chebyshev inequality results in
	\begin{equation}
		\frac{\triangle_{G_n}(B_n(\varepsilon))}{\Exp{\triangle_{G_n}(B_n(\varepsilon))}}\plim 1.
	\end{equation}
	Combining this with Lemma~\ref{lem:expcond} shows that
	\begin{equation}
		\begin{aligned}[b]
		\frac{\triangle_{G_n}(B_n(\varepsilon))}{n^{\frac{3}{2}(3-\tau)}}
		&=(1+\op(1))\mu^{-\frac{3}{2}(\tau-1)} \frac{(C(\tau-1))^3}{6}\int_{\varepsilon}^{1/\varepsilon}\int_{\varepsilon}^{1/\varepsilon}\int_{\varepsilon}^{1/\varepsilon}(xyz)^{-\tau}\\
		& \quad \times \frac{xy}{1+xy}\frac{xz}{1+xz}\frac{yz}{1+yz}\dd x\dd y \dd z.
		\end{aligned}
	\end{equation}
	Thus,
	\begin{align}
	&\frac{T(G_n)}{n^{\frac{3}{2}(3-\tau)}} = 	\frac{\triangle_{G_n}(B_n(\varepsilon))+\triangle_{G_n}(\bar{B}_n(\varepsilon))}{n^{\frac{3}{2}(3-\tau)}}\nonumber\\
	&=(1+\op(1)) \mu^{-\frac{3}{2}(\tau-1)} \frac{(C(\tau-1))^3}{6}\int_{\varepsilon}^{1/\varepsilon}\int_{\varepsilon}^{1/\varepsilon}\int_{\varepsilon}^{1/\varepsilon}(xyz)^{-\tau}\frac{xy}{1+xy}\frac{xz}{1+xz}\frac{yz}{1+yz}\dd x\dd y \dd z\nonumber\\
	& \quad  +o(h(\varepsilon)).
	\end{align}
	Taking the limit of $\varepsilon\downarrow 0$ then proves the convergence in probability statement of Theorem~\ref{thm:triang}.
	
	To prove that the resulting integral is finite, we use that
	\begin{equation}
	\begin{aligned}[b]
	&\int_{0}^{\infty}	\int_{0}^{\infty}	\int_{0}^{\infty}(xyz)^{-\tau}\frac{xy}{1+xy}\frac{xz}{1+xz}\frac{yz}{1+yz}\dd x\dd y \dd z\\
	&  \leq \int_{0}^{\infty}	\int_{0}^{\infty}	\int_{0}^{\infty}(xyz)^{-\tau}(1-\me^{-xy})(1-\me^{-xz})(1-\me^{-yz})\dd x\dd y \dd z .
	\end{aligned}
	\end{equation}
		By~\cite[Lemma 12]{hofstad2017}, the latter integral is finite, concluding the proof of the theorem.
	\end{proof}
	
	\section{Proof of Theorem~\ref{thm:ck}}\label{sec:proofck}
	The proof of Theorem~\ref{thm:ck} follows similar lines as the proof of Theorem~\ref{thm:triang} and again consists of several steps. We first prove Theorem~\ref{thm:ck} for Ranges I, II and IV, and then we show how to adapt the proof for $k=\lceil B\sqrt{n}\rceil$. Let $\triangle_k$ denote the number of triangles where at least one of the vertices has degree $k$. We specify a set $W\subseteq [n]\times [n]$ of ordered pairs of vertices such that the contribution to $\triangle_k$ from triangles where the other two vertices in the triangle fall outside of $W$ is small. We then focus on the number of triangles where one vertex has degree $k$ and the other two vertices are in $W$. We compute the expected number of such triangles, and then use a second moment method to show that the number of such triangles concentrates around its expectation.
	
	More specifically, let 
	\begin{equation}\label{eq:Wnk}
		W_n^k(\varepsilon)=\begin{cases}
		\{(u,v):d_ud_v\in[\varepsilon n,n/\varepsilon]\} & k\leq n^{(\tau-2)/(\tau-1)},\\
		\{(u,v):d_ud_v\in[\varepsilon n,n/\varepsilon], d_u,d_v\leq n/k \} & k>n^{(\tau-2)/(\tau-1)}, k\leq \sqrt{n}, \\
		\{(u,v):d_u,d_v\in[\varepsilon n/k,n/(\varepsilon k)]\} & k> \sqrt{n}.
		\end{cases}
	\end{equation}
	Recall that $N_k=\sum_{i\in[n]}\ind{d_i=k}$ and $X_{uvw}=\ind{\triangle_{u,v,w}}$. Define
	\begin{align}
	c(W_n^k(\varepsilon))& =\frac{2}{N_kk(k-1)}\sum_{(u,v)\in W_n^k(\varepsilon)}\sum_{w:d_w=k}X_{uvw},\\ c(\bar{W}_n^k(\varepsilon))& =\frac{2}{N_kk(k-1)}\sum_{(u,v)\notin W_n^k(\varepsilon)}\sum_{w:d_w=k}X_{uvw}.
	\end{align}
	Thus, $c(W_n^k(\varepsilon))$ denotes the contribution to $c(k)$ from triangles with one vertex of degree $k$ and the other two vertices in $W_n^k(\varepsilon)$, and $c(\bar{W}_n^k(\varepsilon))$ denotes the contribution from all other triangles containing a vertex of degree $k$.

Denote
	\begin{equation}
		f(n,k)=\begin{cases}
		n^{2-\tau}\log(n) & k\leq n^{(\tau-2)/(\tau-1)},\\
		n^{2-\tau}\log(n/k^2) & k>n^{(\tau-2)/(\tau-1)}, k\leq \sqrt{n},  \\
		n^{5-2\tau}k^{2\tau-6} & k> \sqrt{n},
		\end{cases}
	\end{equation}
	which is the scaling of $c(k)$ predicted by Theorem~\ref{thm:ck}. We first bound $c(\bar{W}_n^k(\varepsilon))$, the contribution to $c(k)$ from triangles with degrees outside the ranges in~\eqref{eq:Wnk}:
	
	\begin{lemma}\label{lem:expckminor}
		When $\boldsymbol{d}_n$ satisfies Assumption~\ref{ass:degreesck} for some $\tau\in(2,3)$, then in Ranges I, II and IV,
	\begin{equation}
		\lim_{\varepsilon\downarrow 0}\limsup_{n\to\infty}\frac{\Exp{c(\bar{W}_n^k(\varepsilon))}}{f(n,k)}=0.
	\end{equation}
	\end{lemma}
\begin{proof}
Using~\eqref{eq:pijon}, we can write for some $K_1>0$,
	\begin{align}\label{eq:expckbound}
		&\Exp{c(\bar{W}_n^k(\varepsilon))}= \frac{1}{N_k}\sum_{i:d_i=k}\sum_{(u,v)\notin W_n^k(\varepsilon)}\frac{\Prob{\triangle_{i,u,v}}}{k(k-1)}\nonumber\\
		& \leq \frac{K_1}{N_k k(k-1)}\sum_{i:d_i=k}\sum_{(u,v)\notin W_n^k(\varepsilon)}\min\Big(\frac{kd_u}{\mu n},1\Big)\min\Big(\frac{kd_v}{\mu n},1\Big)\min\Big(\frac{d_ud_v}{\mu n},1\Big)\nonumber \\
		& =\frac{K_1}{k(k-1)} \sum_{(u,v)\notin W_n^k(\varepsilon)}\min\Big(\frac{kd_u}{\mu n},1\Big)\min\Big(\frac{kd_v}{\mu n},1\Big)\min\Big(\frac{d_ud_v}{\mu n},1\Big)\nonumber\\
		& =K_1(n/k)^2\int \int_{H_n(k)}\min\Big(\frac{kx}{\mu n},1\Big)\min\Big(\frac{ky}{\mu n},1\Big)\min\Big(\frac{xy}{\mu n},1\Big)\dd F_n(x)\dd F_n(y),
	\end{align}
	where $H_n(k)\subseteq \mathbb{R}^2$ denotes the region where $(x,y)$ does not satisfy the degree constraints of $W_n^k(\varepsilon)$. We first show that the contribution to the integral from degrees larger than $c n^{1/\tau-1}/\log(n)$ is sufficiently small in Range I. By (\ref{dmax}), $d_{\max}\leq Mn^{1/(\tau-1)}$ for some $M>0$. Therefore, in Range I, $\min(kx/(\mu n),1)=kx/(\mu n)$ for $x\in(0,d_{\max})$, and similarly $\min(ky/L_n,1)=ky/(\mu n)$ for $y\in(0,d_{\max})$. 
	We then write the integral as
	\begin{align}\label{eq:contrdeglarge}
	& \frac{k^2}{(\mu n)^3}\int_{cn^{1/(\tau-1)}/\log(n)}^{Mn^{1/(\tau-1)}}\int_{0}^{n/x}x^2y^2\dd F_n(y)\dd F_n(x)\nonumber\\
	& + \frac{k^2}{(\mu n)^2}\int_{cn^{1/(\tau-1)}/\log(n)}^{Mn^{1/(\tau-1)}}\int_{n/x}^{Mn^{1/(\tau-1)}}xy\dd F_n(y)\dd F_n(x).
	\end{align} 
	We first investigate the contribution of the first integral. By Assumption~\ref{ass:degrees}\ref{ass:degreeall},
	\begin{align}\label{eq:inttemp}
	\int_{0}^{n/x}y^2\dd F_n(y) & = -y^2\big[(1-F_n(y))\big]_0^{n/x}+\int_{0}^{n/x}2y[1-F_n(y)]\dd y\leq  K\int_{0}^{n/x}2y^{2-\tau}\dd y,
	\end{align}
	so that
	\begin{align}
	\frac{k^2}{(\mu n)^3}\int_{cn^{1/(\tau-1)}/\log(n)}^{Mn^{1/(\tau-1)}}\int_{0}^{n/x}x^2y^2\dd F_n(y)\dd F_n(x)
 &\leq 2K \frac{k^2}{n^3}\int_{cn^{1/(\tau-1)}/\log(n)}^{Mn^{1/(\tau-1)}}\int_{0}^{n/x}x^2y^{2-\tau}\dd y\dd F_n(x)\nonumber\\
	& =\frac{2Kk^2}{(3-\tau)n^\tau}\int_{cn^{1/(\tau-1)}/\log(n)}^{Mn^{1/(\tau-1)}}x^{\tau-1}\dd F_n(x).
	\end{align}

	Similarly to~\eqref{eq:inttemp},
	\begin{align}
	&k^2n^{-\tau}\int_{cn^{1/(\tau-1)}/\log(n)}^{Mn^{1/(\tau-1)}}x^{\tau-1}\dd F_n(x)  \nonumber\\
	& =-  k^2n^{-\tau} x^{\tau-1}[1-F_n(x)]_{cn^{1/(\tau-1)}/\log(n)}^{Mn^{1/(\tau-1)}}\!\!+ (\tau-1)k^2n^{-\tau}\!\!\int_{cn^{1/(\tau-1)}/\log(n)}^{Mn^{1/(\tau-1)}}x^{-1}[1-F_n(x)]\dd x \nonumber\\
	& \leq k^2n^{-\tau}\Big[K\Big(\frac{cn^{1/(\tau-1)}}{\log(n)}\Big)^{\tau-1}\Big(\frac{cn^{1/(\tau-1)}}{\log(n)}\Big)^{1-\tau}\Big]+ K(\tau-1)k^2n^{-\tau}\int_{cn^{1/(\tau-1)}/\log(n)}^{Mn^{1/(\tau-1)}}x^{-1}\dd x \nonumber\\
	& = K(\tau-1)k^2n^{-\tau}\int_{cn^{1/(\tau-1)}/\log(n)}^{Mn^{1/(\tau-1)}}x^{-1}\dd x +O(k^2n^{-\tau})\nonumber\\
	& =(\tau-1) k^2n^{-\tau}\log(M\log(n)/c) +O(k^2n^{-\tau}).
	\end{align}
	Multiplying by $n^{2}k^{-2}$ then shows that the contribution of this integral to~\eqref{eq:expckbound} is smaller than $n^{2-\tau}\log(n)$, as required.
	A similar computation shows that the contribution from the second integral in~\eqref{eq:contrdeglarge} is $O(n^{2-\tau}\log(\log(n)))$ as well.	
	To show that the contribution to~\eqref{eq:expckbound} from vertices with degrees larger than $n^{1/(\tau-1)}/\log(n)$ is sufficiently small in Ranges II, III and IV, we use similar computations. 
	
	Let $H_n'(k)\subseteq H_n(k)$ denote the region where $(x,y)$ does not satisfy the degree constraints of $W_n^k(\varepsilon)$, while also $x,y\leq cn^{1/(\tau-1)}/\log(n)$. By our previous arguments, we only need to show that the integral in~\eqref{eq:expckbound} over $H'_n(k)$ is sufficiently small. Note that the integrand is always a power of $x$,$y$ and $z$, where the power depends on the range of the integral. Furthermore, Assumption~\ref{ass:degreesck}\ref{ass:degreerangeck} holds over the entire range of the integral.
	 By Assumption~\ref{ass:degreesck}\ref{ass:degreerangeck}, for $a\ll b\leq c n^{1/(\tau-1)}\log(n)$, $a^\gamma(1-F_n(a))-b^\gamma(1-F_n(b))\leq 0$ for $n$ sufficiently large and $\gamma>1-\tau$. Then, for $n$ sufficiently large and $\gamma>1-\tau$,
	 \begin{align}
	 \int_{a}^{b}x^\gamma \dd F_n(x)& = \big[-x^\gamma(1-F_n(x))\big]_{a}^{b}+\gamma \int_{a}^{b}x^{\gamma-1} (1-F_n(x))\dd x\nonumber\\
	 & \leq \gamma K \int_{a}^{b}x^{\gamma-\tau}\dd x.
	 \end{align}
	 Following the computations in~\cite[Lemma 10]{hofstad2017c}, we can see that indeed the integral over $H'_n(k)$ splits into integrals over ranges $[a,b]$ such that $a\ll b$ and where the integrand equals $x^\gamma$ for some $\gamma>\tau-1$. 
	 Therefore,	we can bound the integral over $H'_n(k)$ for $n$ sufficiently large as
	\begin{align}\label{eq:expckbound2}
	&\int \int_{H_n'(k)}\min\Big(\frac{kx}{\mu n},1\Big)\min\Big(\frac{ky}{\mu n},1\Big)\min\Big(\frac{xy}{\mu n},1\Big)\dd F_n(x)\dd F_n(y)\nonumber\\
	& \leq K_3	\int \int_{H_n'(k)}(xy)^{-\tau}\min\Big(\frac{kx}{\mu n},1\Big)\min\Big(\frac{ky}{\mu n},1\Big)\min\Big(\frac{xy}{\mu n},1\Big)\dd x\dd y
	\end{align}
	for some $K_3>0$.	
	This is the same upper bound used in~\cite[Lemma 10]{hofstad2017c}, so that we can follow the proof of~\cite[Lemma 10]{hofstad2017c} to prove the lemma.	
\end{proof}

We now investigate the expected contribution of vertices in $W_n^k(\varepsilon)$ to $c(k)$:
%
\begin{lemma}\label{lem:expckint}
	Let $\boldsymbol{d}_n$ satisfy Assumption~\ref{ass:degreesck} for some $\tau\in(2,3)$. Let 
	\begin{equation}
		A(\varepsilon)= \int_{\varepsilon}^{1/\varepsilon}\frac{t^{2-\tau}}{1+t}\dd t.
	\end{equation}
	Then,
	\begin{equation}
		\frac{\Exp{c(W_n^k(\varepsilon))}}{f(n,k)}\to
		\begin{cases}
		(C(\tau-1))^2\mu^{-\tau}\frac{3-\tau}{\tau-1}A(\varepsilon)& \text{for }k \text{ in Range I,}\\
		(C(\tau-1))^2\mu^{-\tau}A(\varepsilon) & \text{for } k \text{ in Range II,}\\
		(C(\tau-1))^2\mu^{3-2\tau}\frac{3-\tau}{\tau-1}A(\varepsilon)^2 & \text{for }k \text{ in Range IV.}
		\end{cases}
	\end{equation}
\end{lemma}
\begin{proof}
	Let $v$ be a uniformly chosen vertex of degree $k$. We first investigate the case where $k$ is in Range I. When $i,j\in W_n^k(\varepsilon)$,  $d_id_j=O(n)$ and $d_i,d_j\gg 1$. Furthermore, by Assumption~\ref{ass:degrees}\ref{ass:degreeall}, $d_{\max}=O(n^{1/(\tau-1)})$, so that also $kd_i=o(n)$ and $kd_j=o(n)$ in Range I. Thus, we may use~\eqref{eq:pijon}, so that we can approximate the probability that triangle $i,j,k$ is present as in~\eqref{eq:pijk}. We then obtain
	\begin{equation}\label{eq:expecmajor}
	\begin{aligned}[b]
	\Exp{c(W_n^k(\varepsilon))}&=\frac{1}{k^2}\sum_{(i,j)\in W_n^k(\varepsilon)}\Prob{\triangle_{i,j,v}}\\
	&=\frac{1}{k^2}\sum_{(i,j)\in W_n^k(\varepsilon)}\frac{d_id_j}{d_id_j+L_n}\frac{d_ik}{d_ik+L_n}\frac{d_jk}{d_jk+L_n}(1+o(1)).
	\end{aligned}
	\end{equation}
	In Range I, when $(i,j)\in W_n^k(\varepsilon)$, $d_ik/L_n=o(1)$ so that $d_ik/(d_ik+L_n)=d_ik(1+o(1))$ and similarly $d_jk/(d_jk+L_n)=d_jk/L_n(1+o(1))$. Thus, in Range I,
		\begin{equation}\label{eq:expecmajorI}
	\Exp{c(W_n^k(\varepsilon))}=\sum_{(i,j)\in W_n^k(\varepsilon)}\frac{1}{L_n^2}\frac{d_i^2d_j^2}{d_id_j+L_n}(1+o(1)).
	\end{equation}
	We now analyze the convergence of this expression similarly as in the proof of~\cite[Lemma 6]{hofstad2017c}. The only differences with the expression in~\cite[Lemma 6]{hofstad2017c} are that~\cite[Lemma 6]{hofstad2017c} contains the term $1-\me^{-d_id_j/L_n}$, whereas~\eqref{eq:expecmajorI} contains the term $d_id_j/(d_id_j+L_n)$ instead, which is also a bounded, continuous function. Furthermore, in~\cite[Lemma 6]{hofstad2017c} the degree sequence satisfies Assumption~\ref{ass:degrees}\ref{ass:degreerange} with high probability, changing convergence in probability of~\cite[Lemma 6]{hofstad2017c} to a deterministic limit, similarly as in~\eqref{eq:measconv}. Thus, a similar analysis as in~\cite[Lemma 6]{hofstad2017c} then results in
	\begin{equation}
	\frac{\Exp{c(W_n^k(\varepsilon))}}{f(n,k)}\to (C(\tau-1))^2\mu^{-\tau}\frac{3-\tau}{\tau-1}A(\varepsilon),
	\end{equation}
	which concludes the proof for Range I.
	
	Similarly, in Range II we analyze~\eqref{eq:expecmajor} using~\cite[Lemma 7]{hofstad2017c}, again replacing the function $1-\me^{-xy/L_n}$ by $xy/(L_n+xy)$ and convergence in probability by convergence. Then applying~\cite[Lemma 7]{hofstad2017c} with this replaced function proves the lemma for Range II.
	
	Finally, in Range IV, when $(i,j)\in W_n^k(\varepsilon)$, $d_id_j=o(n)$ so that $d_id_j/(d_id_j+L_n)=d_id_j/L_n^{-1}(1+o(1))$. Thus,~\eqref{eq:expecmajor} becomes
	\begin{equation}\label{eq:expecmajorIII}
	\Exp{c(W_n^k(\varepsilon))}=\frac{1}{k^2L_n}\sum_{(i,j)\in W_n^k(\varepsilon)}\frac{d_i^2k}{d_ik+L_n}\frac{d_j^2k}{d_jk+L_n}(1+o(1)).
	\end{equation}
	Now applying~\cite[Lemma 8]{hofstad2017c}, replacing the function $1-\me^{-xy/L_n}$ by $xy/(L_n+xy)$ and convergence in probability by convergence again proves the lemma, now for Range IV.
\end{proof}

Finally, we show that the variance of $c(k)$ is small in the major contributing regime:

\begin{lemma}\label{lem:varck}
	When $\boldsymbol{d}_n$ satisfies Assumption~\ref{ass:degreesck} for some $\tau\in(2,3)$, then in Ranges I, II and IV,
	\begin{equation}
		\frac{\Var{c(W_n^k(\varepsilon))}}{\Exp{c(W_n^k(\varepsilon))}^2}\plim 0.
	\end{equation}
\end{lemma}
\begin{proof}
	We can write the variance as
	\begin{equation}\label{eq:varckeq}
	\Var{c(W_n^k(\varepsilon))} = \frac{1}{k^4N_k^2}\sum_{i,j:d_i,d_j=k}\sum_{(u,v),(w,z)\in W_n^k(\varepsilon)}\Prob{\triangle_{i,u,v}\triangle_{j,w,z}}-\Prob{\triangle_{i,u,v}}\Prob{\triangle_{j,w,z}}.
	\end{equation}
	Again, the contribution of the summand depends on the size of $\{i,j,u,v,w,z\}$. We first investigate the case where all 6 indices are different and we denote the contribution by this term as $V^{\sss{(6)}}$. Let $g(x,y,z)$ be the function defined in~\eqref{eq:gfunc}. Since $kd_i$, $kd_j$ and $d_id_j$ are all $O(n)$ when $(i,j)\in W_n^k(\varepsilon)$ as well as $d_i,d_j\gg 1$, we may apply~\eqref{eq:pijon} to obtain
	\begin{align}
		\Prob{\triangle_{i,u,v}}\Prob{\triangle_{j,w,z}}&=g(k,d_u,d_v)g(k,d_w,d_z)(1+o(1))\label{eq:triangtwo}\\
		\Prob{\triangle_{i,u,v}\triangle_{j,w,z}}&=g(k,d_u,d_v)g(k,d_w,d_z)(1+o(1)).\label{eq:triangdouble}
	\end{align}
Note that when $d_i=d_j=k$ and $(u,w),(v,z)\in W_n^k(\varepsilon)$, $g(d_i,d_u,d_v)\in[\varepsilon^2\tilde{f}(n,k),\tilde{f}(n,k)/\varepsilon^2]$ for some function $\tilde{f}(n,k)$ depending on the range of $k$. Therefore, the $o(1)$ terms are uniform in $i,j$ with $d_i=d_j=k$ and $(u,v),(w,z)\in W_n^k(\varepsilon)$.
	We then obtain
	\begin{equation}
		\begin{aligned}
		 V^{\sss{(6)}} &= \frac{1}{k^4N_k^2}\sum_{i,j:d_i,d_j=k}\sum_{(u,v),(w,z)\in W_n^k(\varepsilon)}o\left(g(k,d_u,d_v)g(k,d_w,d_z)\right)=o\left(\Exp{c(W_n^k(\varepsilon))}^2\right),
		\end{aligned}
	\end{equation}
	where the last equality follows from~\eqref{eq:expecmajor}.
	
	When $\{i,j,u,v,w,z\}=5$, there are no overlapping edges of the two triangles involved. Thus, we can use the same estimates as in~\eqref{eq:triangtwo} and~\eqref{eq:triangdouble} to show that this contribution is small.
	We now bound the contributions to~\eqref{eq:varckeq} from $\{i,j,u,v,w,z\}\leq 4$. By Lemma~\ref{lem:expckint} we have to show that these contributions are $o(f(n,k)^2)$. We use that, by~\eqref{eq:bound},
	\begin{equation}
	\begin{aligned}[b]
	\Prob{\triangle_{i,u,v}\triangle_{j,w,z}}& -\Prob{\triangle_{i,u,v}}\Prob{\triangle_{j,w,z}}  \leq \Prob{\triangle_{i,u,v}\triangle_{j,w,z}}\\
	& \leq K_1 \min\left(\frac{d_uk}{L_n},1\right)\min\left(\frac{d_vk}{L_n},1\right)\min\left(\frac{d_ud_v}{L_n},1\right)\\
	& \quad \times\min\left(\frac{d_zk}{L_n},1\right)\min\left(\frac{d_wk}{L_n},1\right)\min\left(\frac{d_wd_z}{L_n},1\right),
	\end{aligned}
	\end{equation}
	for some $K_1>0$. This is the same bound as in~\cite[Eq. (5.61)]{hofstad2017c}, where it is shown that these contributions are $o(f(n,k)^2)$ by using a first moment method. The only difference with our setting is that~\cite[Eq. (5.61)]{hofstad2017c} considers i.i.d.\ degrees sampled from~\eqref{D-tailck}, whereas we assume that the empirical degree distribution converges to~\eqref{D-tailck}. This does not influence the expected value, so that we can follow the proof of~\cite[Lemma 9]{hofstad2017c} to show that these contributions are $o(f(n,k)^2)$, as required.
\end{proof}

We now use the above lemmas to prove Theorem~\ref{thm:ck}:

\begin{proof}[Proof of Theorem~\ref{thm:ck}]
	Fix $\varepsilon>0$. By Lemmas~\ref{lem:expckint} and~\ref{lem:varck},
	\begin{equation}
	\frac{c(W_n^k(\varepsilon))}{f(n,k)}\plim
	\begin{cases}
	(C(\tau-1))^2\mu^{-\tau}\frac{3-\tau}{\tau-1}A(\varepsilon)& \text{for }k \text{ in Range I,}\\
	(C(\tau-1))^2\mu^{-\tau}A(\varepsilon) & \text{for } k \text{ in Range II,}\\
	(C(\tau-1))^2\mu^{3-2\tau}\frac{3-\tau}{\tau-1}A(\varepsilon)^2 & \text{for }k \text{ in Range IV.}
	\end{cases}
	\end{equation}
	Combining this with Lemma~\ref{lem:expckminor} results in
	\begin{equation}\label{eq:ckeps}
		\begin{aligned}[b]
		\frac{c(k)}{f(n,k)}& = \frac{c(W_n^k(\varepsilon))+c(\bar{W}_n^k(\varepsilon))}{f(n,k)}\\
		& \plim \begin{cases}
		(C(\tau-1))^2\mu^{-\tau}\frac{3-\tau}{\tau-1}A(\varepsilon)+O(h(\varepsilon))& \text{for }k \text{ in Range I,}\\
		(C(\tau-1))^2\mu^{-\tau}A(\varepsilon) + O(h(\varepsilon))& \text{for } k \text{ in Range II,}\\
		(C(\tau-1))^2\mu^{3-2\tau}\frac{3-\tau}{\tau-1}A(\varepsilon)^2 +O(h(\varepsilon))& \text{for }k \text{ in Range IV.}
		\end{cases}
		\end{aligned}
	\end{equation}
	We then take the limit of $\varepsilon\downarrow 0$. By~\cite[Eq. 3.194.3]{integrals2015}
	\begin{equation}
	\lim_{\varepsilon\downarrow 0}A(\varepsilon)=\int_0^\infty \frac{t^{2-\tau}}{1+t}\dd t =-\frac{\pi}{\sin(\pi(2-\tau))}=\frac{\pi}{\sin(\pi \tau)},
	\end{equation} 
	which is equal to the constant $A$ of Theorem~\ref{thm:ck}. Therefore, taking the limit of $\varepsilon\downarrow 0$ in~\eqref{eq:ckeps} results in
	\begin{equation}
			\frac{c(k)}{f(n,k)}\plim \begin{cases}
			(C(\tau-1))^2\mu^{-\tau}\frac{3-\tau}{\tau-1}A& \text{for }k \text{ in Range I,}\\
			(C(\tau-1))^2\mu^{-\tau}A & \text{for } k \text{ in Range II,}\\
			(C(\tau-1))^2\mu^{3-2\tau}\frac{3-\tau}{\tau-1}A^2 & \text{for }k \text{ in Range IV,}
			\end{cases}
	\end{equation}
	which proves the theorem for Ranges I,II and IV

We now analyze $c(k)$ for Range III where $k=\lceil B\sqrt{n}\rceil$. Note that Lemmas~\ref{lem:boundcontr} and~\ref{lem:varck} also hold for $k=\lceil B\sqrt{n}\rceil$, following the proofs for $c(k)$ in Range IV. We therefore only need to analyze $\Exp{c(W_n^{B\sqrt{n}}(\varepsilon))}$. Using~\eqref{eq:expecmajor} yields
\begin{equation}
\begin{aligned}[b]
\Exp{c(W_n^{B\sqrt{n}}(\varepsilon))}
&=(1+o(1))\frac{1}{B^2n}\sum_{(i,j)\in W_n^k(\varepsilon)}\frac{d_id_j}{d_id_j+\mu n}\frac{d_iB\sqrt{n}}{d_iB\sqrt{n}+\mu n}\frac{d_jB\sqrt{n}}{d_jB\sqrt{n}+\mu n}.
\end{aligned}
\end{equation}
Define the measure 
\begin{equation}
\Nn([a,b])=\frac{(\mu\sqrt{n})^{\tau-1}}{n} \sum_{i\in[n]}\ind{d_i\in\mu/\sqrt{n} [a,b]}.
\end{equation}
Then, 
\begin{align}
\Exp{c(W_n^{B\sqrt{n}}(\varepsilon))}
&=(1+o(1))\frac{\mu^{2-2\tau}n^{3-\tau}}{B^2n}\int_\varepsilon^{1/\varepsilon} \int_\varepsilon^{1/\varepsilon} \frac{xy}{xy+\mu^{-1}}\frac{xB}{xB+1}\frac{yB}{yB+1}\dd\Nn(x) \dd \Nn(y)\nonumber\\
& = (1+o(1))\mu^{2-2\tau}n^{2-\tau}\int_\varepsilon^{1/\varepsilon}\int_\varepsilon^{1/\varepsilon} \frac{xy}{xy+\mu^{-1}}\frac{x}{xB+1}\frac{y}{yB+1}\dd\Nn(x) \dd \Nn(y).
\end{align}
Using similar techniques as in~\eqref{eq:gconv}, we can prove that
\begin{equation}
\begin{aligned}[b]
\frac{\Exp{c(W_n^{B\sqrt{n}}(\varepsilon))}}{\mu^{2-2\tau}n^{2-\tau}}
&\to (C(\tau-1))^2\!\!\int_\varepsilon^{1/\varepsilon}\!\!\int_\varepsilon^{1/\varepsilon}\!\! \frac{(t_1t_2)^{2-\tau}}{t_1t_2+\mu^{-1}}\frac{1}{t_1B+1}\frac{1}{t_2B+1}\dd t_1 \dd t_2.
\end{aligned}
\end{equation}
Combining this with Lemmas~\ref{lem:boundcontr} and~\ref{lem:varck} and taking the limit of $\varepsilon\downarrow 0$ then proves the theorem for $k=\lceil B\sqrt{n}\rceil$.
\end{proof}

\section{Conclusion}\label{sec:conc}
In this paper, we have studied the number of triangles in uniform random graphs with given degrees, when the degree sequence follows a power-law distribution with degree-exponent $\tau\in(2,3)$. We have shown that the rescaled number of triangles converges in probability to a constant. We have further shown that most triangles occur on vertices of degrees proportional to $\sqrt{n}$. Another interesting conclusion is that uniform random graphs asymptotically contain less triangles than the erased configuration model which is often used to approximate uniform random graphs with a given degree sequence. 

We have also shown that the local clustering coefficient $c(k)$ of uniform random graphs with scale-free degrees behaves differently in three ranges. For small values of $k$, $c(k)$ is independent of $k$. In the second range, $c(k)$ starts to decay slowly in $k$, and in the fourth range $c(k)$ decays as a power of $k$. 

The triangle is an interesting subgraph since it allows one to analyze clustering properties, but counting other subgraphs in uniform graphs with scale-free degrees would also be interesting. We believe that our results easily extend to a wider class of subgraphs, as for the erased configuration model in~\cite[Theorem 2.2]{hofstad2017d}. In this class of subgraphs, most subgraphs are also supported on vertices of degree proportional to $\sqrt{n}$. This class of subgraphs contains for example cliques of any fixed order that is at least three. Extending the results to count subgraphs outside this class would also be an interesting question for further research.

\paragraph{Acknowledgements.}
The work of RvdH and CS was supported by NWO TOP grant 613.001.451.
The work of RvdH is further supported by the NWO Gravitation Networks grant 024.002.003 and the NWO VICI grant 639.033.806. The work of PG is supported by ARC DE170100716 and ARC DP160100835. The work of AS is supported by an Australian Government Research Training Program Scholarship.

	\bibliographystyle{abbrv}
\bibliography{references}
	
\end{document}